\documentclass[oneside]{amsart}

\usepackage[letterpaper,body={14.6cm, 22.5cm}, mag=1000]{geometry}
\usepackage{amssymb}
\usepackage{amsthm}
\usepackage{amscd}
\usepackage{lineno}
\usepackage{amsmath}
\numberwithin{equation}{section}
\theoremstyle{plain}
\newtheorem{thm}{Theorem}[section]
\newtheorem{cor}[thm]{Corollary}
\newtheorem{lemma}[thm]{Lemma}

\newtheorem{prop}[thm]{Proposition}

\newtheorem*{thma}{Theorem A}
\newtheorem*{thmb}{Theorem B}

\theoremstyle{definition}

\newcommand{\dlabel}[1]{\ifmmode \text{\ttfamily \upshape [#1] } \else
{\ttfamily \upshape [#1] }\fi \label{#1} }

\newcommand{\C}{\operatorname{C} }

\newcommand{\K}{\operatorname{K} }

\newcommand{\Z}{\operatorname{Z} }

\newcommand{\gen}[1]{\left < #1 \right >}

\begin{document}
\baselineskip 15pt

\title{Commutators in groups of order $p^7$ }

\author{Rahul Kaushik}

\address{School of Mathematics, Harish-Chandra Research Institute, HBNI,
Chhatnag Road, Jhunsi, Allahabad - 211 019, INDIA}

\email{rahulkaushik@hri.res.in}

\author{Manoj K.~Yadav}

\address{School of Mathematics, Harish-Chandra Research Institute, HBNI, 
Chhatnag Road, Jhunsi, Allahabad - 211 019, INDIA}

\email{myadav@hri.res.in}

\subjclass[2010]{Primary 20D15, 20F12}
\keywords{commutator, commutator subgroup, $p$-group}
\begin{abstract}
We present a characterisation of groups $G$ of order $p^7$, $p$ prime,  in which  not all elements of the commutator subgroup $\gamma_2(G)$ of $G$ are commutators in  $G$. On the way we obtain several structural results on groups of order $p^7$. 
\end{abstract}

\maketitle

\section{Introduction}	
	
	Let $G$ be a finite group and $\K(G) := \{[a,b] \mid a, b \in G \}$.  It is well known that the commutator subgroup $\gamma_2(G)$  of $G$ is generated by $\K(G)$.  A natural question that has attracted the attention of many mathematicians over last many many decades is whether $\gamma_2(G)$  is equal to $\K(G)$ or not for groups $G$ in a given class of groups. In 1899, G.A. Miller \cite{gM99} proved that $\K(G) = \gamma_2(G)$ for all alternating groups $G = A_n$, $n \ge 5$. This result was reinvented by N. Ito \cite{nI51} and O. Ore  \cite{oO51} in 1951, and Ore  conjectured that the same is true for all finite simple groups. With the efforts of several people over the years, the conjecture was proved in 2010 \cite{LOST10}.  It is also true that $\K(G) = \gamma_2(G)$ for most of finite quasi-simple groups \cite{LOST11}. Furthermore, the commutator length of all finite quasi-simple groups is  at most $2$. The commutator length of a finite group $G$ is the smallest positive integer $m$ such that  every element of  $\gamma_2(G)$ can be written as a product of at most $m$ commutators in $G$.
	
Interest in the above question for non-perfect groups  also goes back to 1902 when W.B. Fite \cite{wF02} proved that if a finite $p$-group $G$ of nilpotency class $2$  is such that $G/\Z(G)$ is minimally generated by three elements, then $\K(G) = \gamma_2(G)$.  He also presented examples of $2$-groups $G$ of nilpotency class $2$ such that  $G/\Z(G)$ is minimally generated by four elements and $\K(G) \ne \gamma_2(G)$.  The smallest such $2$-group is of order $2^8$. There do exist such groups of order $p^8$ for all odd primes \cite[Theorem A]{MR}. In 1953, K. Honda \cite{kH53} re-invented, in a slightly different form,  the famous character sum formula of F.G. Frobenius for determining whether a given element of a finite group is a commutator or not. This result was used to prove that if, for a finite group $G$, $\gamma_2(G)$ is generated by a commutator, then each generator of $\gamma_2(G)$ is a commutator.  D.M.  Rodney  \cite{dmR73,  dmR77}  proved that  if a nilpotent group $G$ has cyclic commutator subgroup,  or if  $G$ is finite and  $\gamma_2(G)$ is elementary abelian of order $p^3$ for any prime integer $p$,  then $\K(G) = \gamma_2(G)$.  This result was further generalised by  R. Guralnick \cite{rmG82}, who proved that if $G$ is a  finite group whose  $\gamma_2(G)$  is an abelian $p$-group minimally generated by at most $3$ elements,  where $p \geq 5$,  then $\K(G) = \gamma_2(G)$.   

Coming to $p$-groups, Fern\'{a}ndez-Alcober and De las Heras \cite{FAH19}  proved that $\K(G) = \gamma_2(G)$ for  finite $p$-group $G$ such that $\gamma_2(G)$ is generated by $2$ elements.  This result was further extended  by De las Heras \cite{iH20} for  finite $p$-groups $G$,  $p \geq 5$,  such that $\gamma_2(G)$ is generated by $3$ elements. No more generalisation is possible in terms of the rank of the commutator subgroup of  finite $p$-groups.    Kappe and Morse  \cite{KM05},  constructed examples of group $G$ of order $p^6$, $p$ odd, such that $\gamma_2(G)$ is elementary of order $p^4$, but $\K(G) \ne \gamma_2(G)$. In the same paper, they proved that $\K(G) = \gamma_ 2(G)$ for all $p$-groups of order at most $p^5$ and for all $2$-groups of order at most $2^6$.  They  also exhibited groups of  order $2^7$ for which $\K(G) \neq \gamma_2(G)$.  Classification of finite $p$-groups $G$ (upto isoclinism) with  $\gamma_2(G)$ of order  $p^4$ and exponent $p$ such that  $\K(G) \ne \gamma_2(G)$ is provided in \cite{MR} by the present authors. As a result, a clear characterisation of groups $G$ of order $p^6$ such that $\K(G) \ne \gamma_2(G)$ now exists.  These are the groups which fall in the isoclinism class $\Phi(23)$. As classified in \cite{rJ80}, there are total $43$ isoclinism classes of groups of order $p^6$, $p$ odd prime.  A detailed historical survey on the above mentioned  question  is available in \cite{KM07}.

Apart from these results, to the best of our knowledge, nothing is known about finite $p$-groups $G$ with $\K(G) = \gamma_ 2(G)$. And the theory developed so far seems revealing no patterns.  So it is highly desirable to investigate more classes of finite $p$-groups.   In this article we take up this investigation for groups of order $p^7$,  and prove the following result. 
\begin{thma}\label{mainthm}
Let $G$ be a group of order $p^7$,  $p \geq 5$.  Then the following statements hold:

{\rm (1)} If $d(\gamma_2(G)) \le 3$ or the nilpotency class of $G$ is $6$, then $\K(G) = \gamma_2(G)$.

{\rm (2)}  If $|\gamma_2(G)| = p^4$ and $d(\gamma_2(G)) = 4$, then $\K(G) \neq \gamma_2(G)$ if and only if  the nilpotency class of $G$ is either $3$ or $4$ and $|\Z(G)| = p^3$.

{\rm (3)} If the nilpotency class of $G$ is $5$, $|\gamma_2(G)| = p^5$, $d(\gamma_2(G)) = 4$ and $\exp(\gamma_2(G)) = p^2$, then $\K(G) \neq \gamma_2(G)$.

{\rm (4)}  If the nilpotency class of $G$ is $4$, $|\gamma_2(G)| = p^5$, $d(\gamma_2(G)) \ge 4$, then $\K(G) \neq \gamma_2(G)$ if and only if  there exists a subgroup $H \leq Z(G)$ of order $p$ such that $|Z(G/H)|=p^2$.

{\rm (5)}  If the nilpotency class of $G$ is $5$ and $\gamma_2(G)$ is elementary of order  $p^5$,    then $\K(G) \neq \gamma_2(G)$.

{\rm (6)}  If the nilpotency class of $G$ is $5$ and $\gamma_2(G)$ is non-abelian of order $p^5$ and exponent $p$,  then $\K(G) \neq \gamma_2(G)$ if and only if  $|\gamma_5(G)| = p$ and  $|\Z_2(G)| = p^3$.

Moreover,   the commutator  length of $G$ is at most $2$.   
   
\end{thma}

While proving this theorem, we obtain many interesting structural results on groups of order $p^7$. 
We use GAP \cite{GAP} for handing the remaining cases of $p$, that is,  $2$ and $3$, and derive the following conclusion.

\begin{thmb}
For a group $G$ of order $2^7$, $\K(G) \neq \gamma_2(G)$ if and only if $G$ is of nilpotency class $3$,  $\gamma_2(G)$ is a  $3$-generated abelian subgroup of order $2^4$  and $|Z(G)|=2^3$.  

Let $G$ be a group of order $3^7$.  Then $K(G) \neq \gamma_2(G)$ if and only if one of the following holds:

{\rm (i)}  The  nilpotency class of $G$ is $5$,  $\gamma_2(G)$ is a $3$-generated non-abelian subgroup,  $|\gamma_5(G)| = 3$ and $|\Z_2(G)| = 3^3$. 

{\rm (ii)}  The  nilpotency class of $G$ is $5$, $\gamma_2(G)$ is a $3$-generated abelian subgroup of order $3^5$ and exponent $3^2$  and $|\Z(G)| \leq 3^2$.

{\rm (iii)}  The  nilpotency class of $G$ is $4$, $\gamma_2(G)$ is a $3$-generated abelian subgroup of order $3^4$ and $|\Z(G)| = 3^3$. 

{\rm (iv)}  The  nilpotency class of $G$ is $4$,  $\gamma_2(G)$ is a $4$-generated  abelian  subgroup of exponent $3^2$. 

{\rm (v)}    The  nilpotency class of $G$ is $3$,  $\gamma_2(G)$ is a $4$-generated elementary abelian subgroup and $|\Z(G)| = 3^3$.

Moreover,   the commutator  length of $G$ is at most $2$. 
\end{thmb}

Let  $1 \to H \to E \to G \to 1$ be a group extension.  Notice that if $\K(G) \ne \gamma_2(G)$, then $\K(E) \ne \gamma_2(E)$ (see Lemma \ref{quotient}). But the converse is not true in general. Indeed, there do exist groups $E$ such that $\K(E) \ne \gamma_2(E)$, but all proper quotient groups $E/H$ of $E$ satisfy $\K(E/H) = \gamma_2(E/H)$. All groups of  order $p^6$ from $\Phi(23)$ are such.  It follows that the  groups $G$ of nilpotency class $3$ with $|\Z(G)| = p^3$ occurring in Theorem A (2) are also such groups. But, as it follows from the results below,  all other  groups $G$ of order $p^7$ with $\K(G) \ne \gamma_2(G)$ are central extensions of groups of order $p^6$ (from the isoclinism class $\Phi(23)$) by a subgroup of order $p$.  On the basis of the results of this paper and other known results in the literature,  the following seems to hold for all finite $p$-groups $G$ : 
  {\it $\K(G) = \gamma_2(G)$ if and only if $\Z(G/N) \cap \gamma_2(G/N) \subseteq \K(G/N)$ for all normal subgroups $N$ of $G$ (including the trivial one)}.

A brief layout of the paper is as follows. Section $2$   contains  prerequisites, including some number theoretic observations. Groups of  nilpotency class $5$ under consideration are studied in Section $3$, and that of class $4$ in Section $4$. Proof of Theorem  A is presented in Section $5$, where a variety of examples illustrating the theory are also presented. 

We conclude this section by setting some notations.  For two elements $x$ and $y$ of a group $G$,  $x^y$ denotes  the conjugate $y^{-1}xy$ of $x$ by $y$, and $[x, y]$ denotes the commutator $x^{-1}y^{-1}xy$  of $x$ and $y$ in $G$.  Left normed higher  commutators  are defined iteratively, for example $[x, y, z] := [[x, y], z]$. Centraliser of an element  $x \in G$ (a subgroup $H$ of $G$) is denoted by
$\C_G(x)$  ($\C_G(H)$). Terms of the lower and upper central series of $G$ are denoted by $\gamma_i(G)$ and $\Z_i(G)$ respectively.  For a finite $p$-group $G$, minimally generated by two elements, say $\alpha_1$ and $\alpha_2$, the following notations for commutators will be used throughout:  $\beta= [\alpha_1, \alpha_2]$,  $\beta_i= [\beta, \alpha_i]$,  $\eta_{ij}=[\beta_i, \alpha_j]$ and $\xi_{ijk}=[\eta_{ij},\alpha_k]$,  where $i,j,k \in \{1,2\}$. The number of elements in a  minimal  generating set of a finite group $G$ is denoted by  $d(G)$.  

%%%%%%%%%%%%%%%%%%%%%%%%%%%%%%%%%%%%%%%%%%%%%%%%%%%%%%%%%%%%%%%%%%%%%%%%%%%%%%%%%%%%%%	

\section{Prelamanaries}

The following concept of isoclinism of groups is due to P. Hall \cite{pH40}.
Let $X$ be a  group and $\bar{X} := X/\Z(X)$.  Then commutation in $X$ gives a well defined map $a_{X} : \bar{X} \times \bar{X} \mapsto \gamma_{2}(X)$ such that
$a_{X}(x\Z(X), y\Z(X)) = [x,y]$ for $(x,y) \in X \times X$.  Two  groups $G$ and $H$ are called \emph{isoclinic} if  there exists an  isomorphism $\phi$ of the factor group
$\bar G = G/\Z(G)$ onto $\bar{H} = H/\Z(H)$,  and an isomorphism $\theta$ of the subgroup $\gamma_{2}(G)$ onto  $\gamma_{2}(H)$ such that the following diagram is commutative
\[
 \begin{CD}
   \bar G \times \bar G  @>a_G>> \gamma_{2}(G)\\
   @V{\phi\times\phi}VV        @VV{\theta}V\\
   \bar H \times \bar H @>a_H>> \gamma_{2}(H).
  \end{CD}
\]
The resulting pair $(\phi, \theta)$ is called an \emph{isoclinism} of $G$  onto $H$. Notice that isoclinism is an equivalence relation among  groups.  The equivalences classes under isoclinism are called \emph{isoclinism classes}. The following result follows from \cite{pH40}.

\begin{lemma}\label{hlemma}
In each isoclinism class of groups there exists a group $G$ such that $\Z(G) \le \gamma_2(G)$.
\end{lemma}

The following three results are from  \cite[Lemma 2.2, Lemma 2.5, Lemma 4.1]{MR}. 
 \begin{lemma}\label{prelemma}
Let $G$ and $H$ be two isoclinic finite $p$-groups.  Then  $\K(G) = \gamma_2(G)$ if and only if $\K(H) = \gamma_2(H)$.
\end{lemma}

\begin{lemma} \label{prelemma2}
Let $G$ be a finite group and $H\leq \gamma_2(G) \cap Z(G)$. If there exist $x_1,x_2, \ldots , x_n$ such that $\gamma_2(G)/H= \bigcup\limits_{i=1}^{n} [x_i H, G/H] $ and $H\subseteq \bigcap \limits_{i=1}^{n} [x_i, G]$,  then $\gamma_2(G)= \bigcup \limits_{i=1}^{n} [x_i, G]$.
\end {lemma}

\begin{lemma}\label{p6}
Let $G$ be a  group of order $p^6$,  $p \ge 3$ with $\gamma_2(G)$  elementary abelian of order $p^4$ and $|\Z(G)| = p^2$.   Then $\K(G) \neq \gamma_2(G)$. 
\end{lemma}

 We now mention  very elementary but extremely useful observations.
\begin{lemma}\label{quotient}
Let $G$ be a finite group and $H$ a normal subgroup of $G$.  If $\K(G/H) \neq \gamma_2(G/H)$,  then $\K(G) \neq \gamma_2(G)$.
\end{lemma}	
\begin{proof}
Let $\bar{G}= G/H$, and $\bar{x} \in \gamma_2(\bar{G})$ be an element which is not in $\K(\bar{G})$.  Then its pre-image $x \notin \K(G)$.  Hence  $\K(G) \neq \gamma_2(G)$.    
\hfill $\Box$

\end{proof}

\begin{lemma}\label{cmlength}
Let $G$ be a group and $H$ a normal subgroup of $G$ contained in $\K(G)$ such that $\K(G/H)= \gamma_2(G/H)$. Then the commutator length of $G$ is at most $2$.
\end{lemma}

The following basic identities will be used throughout, mostly, without any further reference
\begin{lemma}
Let $x, y , z$ be elements of a group $G$ and $n$ be a positive integer. Then

(i) $[x, yz] = [x, z][x, y]^z$,

(ii) $[xy, z] = [x, z]^y[y, z]$,

(iii) $[x, y^n] = [x, y]^n$  whenever $y$ centralizes $[x, y]$.
\end{lemma}
 
Let $G$ be a $p$-group of maximal class of order $p^n$, that is, the nilpotency class of $G$ is $n-1$.  Then $\C_i := \C_G(\gamma_i(G)/\gamma_{i+2}(G))$ are called \emph{two-step centralizers} in $G$,  where $1 \leq i  \leq n-2$.  It is clear that $\gamma_2(G) \leq \C_i$  for all $1 \leq i  \leq n-2$.  Also,  all $\C_i$ are characteristic maximal subgroups of $G$.  An element $s \in G$ is said to be   \emph{uniform}  if $s  \notin   \bigcup \limits_{i=1}^{n-2} \C_i$.	It was proved by N. Blackburn \cite{NB58} that every finite $p$-group of maximal class admits uniform elements. The following result follows from \cite[III.14.23 Satz]{bH67} (may also see  \cite{MAM}).
\begin{thm}\label{max}
Let $G$ be a finite $p$-group of maximal class.  Then $\K(G)=\gamma_2(G)$. More precisely, $\gamma_2(G) = [s, G]$ for a uniform element $s$ of $G$.
\end{thm}

The following result is from \cite[Theorem 4.7]{FA01} (also see  \cite[III.14.14 Hilfssatz]{bH67}), which also follows from \cite[Theorem 3.2]{NB58} as a special case.
\begin{thm}\label{alc1}
Let $G$ be a $p$-group of maximal class of order at most $p^{p+1}$. Then $\exp(G/\Z(G)) = \exp(\gamma_2(G)) = p$.
\end{thm}

The proof of the following result is left as an easy exercise for the reader. 
\begin{lemma}\label{nab1}
Let $G$ be a group of order $p^7$ and  nilpotency class $4$ with  $|\gamma_2(G)| = p^5$.  Then $\gamma_2(G)$ is abelian.
\end{lemma}

We now mention some number theoretic results.  For a prime integer  $p$ and $d \in \mathbb{F}_p$, the field of $p$ elements,  Legendre symbol, denoted by  $(\frac{d}{p})$,  is defined as 
$$ \left( \frac{d}{p} \right)=  \begin{cases}
0 & \ \ \ \ \ \ d=0,\\
1 & \ \ \ \ \  \text{ d is quadratic residue} \pmod p, \\
-1 &  \ \ \ \ \ \ \text{d is non-quadratic residue} \pmod p.
\end{cases}$$

\begin{lemma}\label{discriminant}
Let $p$ be an odd prime and $f(\lambda, \mu ) =a\lambda^2 + b \lambda \mu + c\mu^2$ be a binary quadratic form over $\mathbb{F}_p^*$.  If $p$ divides $b^2-4ac$,  then $f=0$ has a nontrivial solution in $\mathbb{F}_p^*$.  
\end{lemma}
\begin{proof}
Let $\mu $ be a fixed element of $\mathbb{F}_p^*$.  Then $f= a\lambda^2 + b \lambda \mu + c\mu^2$ will be a quadratic equation in $\lambda$,  whose discriminant is  $(b^2  - 4ac) \mu^2$ is zero in $\mathbb{F}_p$.  Thus $\lambda= -b\mu /2a$ is a nontrivial solution of $f = 0$.     \hfill $\Box$

\end{proof}

We say that a binary quadratic form $f(\lambda, \mu ) =a\lambda^2 + b \lambda \mu + c\mu^2$ represents an integer $r$,  if there exist some integers $\lambda_0, \mu_0$ such that $f(\lambda_0, \mu_0)= r$. The following theorem and its corollary are being proved in \cite[Chapter 3,  Page 153]{IHH}.  
\begin{thm}\label{discthm}
Let $n$ and $d$ be given integers with $n \neq 0$. Then there exist a binary quadratic form of discriminant $d$ that represents $n$ if and only if the congruence $x^2 \equiv d  \pmod {4|n|}$ has a solution.
\end{thm}

\begin{cor}\label{disccor}
Let $d= 0$ or $1 \pmod 4$,  and $p$  be an odd prime.   Then there exists a binary quadratic form of discriminant $d$ that represents $p$ if and only if $(\frac{d}{p})=1$.
\end{cor}

\begin{lemma}\label{quad}
For given $m,n \in \mathbb{F}_p^*$,  let $f(\lambda, \mu ) = m\lambda^2 + n \lambda \mu - \mu^2 $ and  $g(\lambda, \mu) = \lambda^2 + n \lambda \mu - m\mu^2$,  be two quadratic forms,  where  $p$ is an odd prime.  Then $f(\lambda, \mu) = 0 $ admits only a trivial solution in $\mathbb{F}_p$  if and only if $g(\lambda, \mu) = 0$  does the same. 
\end{lemma}
\begin{proof}
Assume that  $f(\lambda, \mu) = 0$ implies $\lambda = \mu =0$ in $\mathbb{F}_p$.  Let $f$ represent $p$;  meaning there exist   $\lambda_0, \mu_0 \in \mathbb{Z}$ such that $f(\lambda_0, \mu_0)=p$.  If both  $\lambda_0$ and $\mu_0$ are congruent to zero modulo $p$,  then $p^2 $ divides $p$ (= $m\lambda_0^2 + n \lambda_0 \mu_0 - \mu_0^2$),   which is not possible. Hence,  it follows that none of  $\lambda_0,  \mu_0 $ is a multiple of $p$.  We can then find  $a_0,  b_0 \in \mathbb{F}_p^*$ such that $f(a_0, b_0) \equiv 0  \pmod p$,  where $\lambda_0 \equiv a_0 \pmod p$ and $\mu_0 \equiv b_0 \pmod p$. This contradicts our assumption.  Thus $f$ does not represent $p$.  By our assumption and Lemma \ref{discriminant}, $n^2+4m \in \mathbb{F}_p^*$,  and since  $n^2+ 4m \equiv 0,1 \pmod 4$,  it follows from Corollary \ref{disccor} that  $\big(\frac{n^2+4m}{p}\big) = -1$.  Hence $x^2 \equiv n^2 +4m  \pmod p$ does not have a solution.  Now let $g(\lambda, \mu)=0$ has a non-trivial solution. Thus $g$ represents $ps$ for some $s \in \mathbb{F}_p$,  hence by Theorem \ref{discthm} it follows that $x^2 \equiv n^2+4m \pmod {4ps}$ has a solution. Thus $x^2 \equiv n^2 +4m  \pmod p$ must have a solution,  which is not possible.  Hence $g(\lambda, \mu)=0$ has only trivial solution.   Converse also follows on the same line.
            \hfill $\Box$

\end{proof}

%%%%%%%%%%%%%%%%%%%%%%%%%%%%%%%%%%%%%%%%%%%%%
%%%%%%%%%%%%%%%%%%%%%%%%%%%%%%%%%%%%%%%%%%%%%

\section{Groups of nilpotency class 5}

In this section we deal with groups $G$ of order $p^7$ and nilpotency class $5$ with $d(\gamma_2(G)) \ge 4$.  As we'll see, for most of these groups $G$, $\K(G) \ne \gamma_2(G)$.

\begin{lemma}\label{cls5}
Let $G$ be a group  of order $p^7$, $p \geq 5$, and  nilpotency class $5$  such that $\gamma_2(G)$ is a $4$-generated subgroup of order $p^5$ and  exponent  $p^2$.  Then the following statements hold:

(i) $|\gamma_5(G)|=p$,

(ii)  $|\Z(G/\gamma_5(G))| = p^2$ and  $|\Z_2(G)| = p^3$,

(iii) $\K(G) \ne \gamma_2(G)$,

(iv) Commutator length of $G$ is $2$.
\end{lemma}
\begin{proof}
Notice that $|\gamma_5(G)| \le p^2$. For the first assertion, contrarily assume that the order of $\gamma_5(G)$ is $ p^2$.   Thus $\gamma_5(G) = \Z(G)$ and $\gamma_i(G)/\gamma_{i+1}(G)$ is of order $p$  for every $i \in \{ 2,3,4\}$.  By the given hypothesis we can assume that $G$ is generated by $2$ elements,  say $\alpha_1, \alpha_2$.   Recall that $\beta = [\alpha_1, \alpha_2]$. If  $\beta^p=1$, then, $[x, \beta, \beta]$ being in $\gamma_5(G) \le \Z(G)$ and $p$ being at least $5$,  for any $x \in G$ we have
  $$1=[x,\beta^p]=[x, \beta]^p.$$ 
 Since $\gamma_3(G)$ is abelian, it now follows that  the exponent of $\gamma_3(G)$ is $p$; hence, being $[\gamma_2(G), \gamma_3(G)]$ central, the exponent of $\gamma_2(G)$ is also $p$, which contradicts the given hypothesis.  Hence $\beta^p \neq 1$.  Since  $d(\gamma_2(G))=4$,  it follows that  $|\Phi(\gamma_2(G))|=p$, which further implies that $\Phi(\gamma_2(G)) \leq \gamma_5(G)$.  Using the fact that $p \ge 5$, it is now easy to see that  the exponent of $\gamma_3(G)$ is $p$. Let $H$ be a complement of  $\Phi(\gamma_2(G))$ in $\gamma_5(G)$. Then the quotient group $\bar{G} := G/H$ is a  $p$-group of maximal class and  order $p^6$  with $\exp(\gamma_2(\bar{G})) = p^2$, which is not possible by Theorem \ref{alc1}.   Hence $\gamma_5(G)$ cannot be of order $p^2$, which estalishes assertion (i).
 
 We now prove assertion (ii). By assertion (i) we have $|\gamma_5(G)| = p$. It then follows from \cite[Theorem 2.12]{NB58} that $|\gamma_4(G)/\gamma_5(G)| = p$, and therefore $|\gamma_3(G)/\gamma_4(G)| = p^2$. If $|\Z(G)| = p^2$, then obviously $|\Z_2(G)| = p^3$ and, since $\gamma_4(G)/\gamma_5(G) \le \Z(G/\gamma_5(G))$, it follows that $|\Z(G/\gamma_5(G))| = p^2$.  So assume that $\Z(G) = \gamma_5(G)$.  Obviously $|\Z_2(G)/\Z(G)| \le p^2$. It suffices to show that $|\Z_2(G)/\Z(G)| \ge p^2$. A regular computation shows that 
 $$[\alpha_1^p, \alpha_2] = [\alpha_1, \alpha_2]^p =  [\alpha_1, \alpha_2^p].$$
 It is easy to see that neither $\alpha_1^p$ nor $\alpha_2^p$ lies in $\gamma_5(G)$. For, if $\alpha_1^p \in \gamma_5(G)$, then $[\alpha_1^p, \alpha_2] = \beta^p = 1$; not true. Use symmetry for $\alpha_2$.
   If  both $\alpha_1^p, \alpha_2^p  \in \gamma_4(G)$, then, $\gamma_4(G)/ \gamma_5(G)$ being  of order $p$,  $\alpha_1^p = (\alpha_2^p)^r$  modulo $\gamma_5(G)$ for some $r \in \mathbb{F}_p^*$, which implies that both $\alpha_1^p$ and $\alpha_2^p$ are central; not possible as just shown.  Hence at least one of $\alpha_1^p$ and $\alpha_2^p$ lies outside $\gamma_4(G)$. Since $[\alpha_1^p, \alpha_2] = [\alpha_1, \alpha_2]^p = \beta^p \in \Z(G)$, it follows that $\alpha_1^p \in \Z_2(G)$. Similarly  $\alpha_2^p \in \Z_2(G)$. Hence $\gen{\gamma_4(G), \alpha_1^p, \alpha_2^p} \le \Z_2(G)$, and therefore $|\Z_2(G)/\Z(G)| \ge p^2$, which proves assertion (ii). 
   
   Assertion (iii) is now clear from assertion (ii), Lemma \ref{p6} and Lemma \ref{quotient}. 
   
   For assertion (iv) we consider two different cases, namely, $\gamma_2(G)$ is abelian or not. If it is non-abelian, then  notice that both  $\C_G(\beta_1)$ and $\C_G(\beta_2)$ can not be maximal subgroups of $G$. Otherwise, $\gamma_2(G)$ will be abelian. Without loss of generality, assume that $\C_G(\beta_1)$ is not maximal. Hence $[\beta_1, G] = \gamma_4(G) \subseteq \K(G)$. Since $\K(G/\gamma_4(G)) = \gamma_2(G/\gamma_4(G))$ from \cite{KM05}, the assertion follows from Lemma \ref{cmlength} and assertion (iii). Now assume that $\gamma_2(G)$ is abelian. Then it is not difficult to see that $[\alpha_i, \gamma_2(G)]$ are  normal subgroups of  $G$ for $i = 1, 2$. Also, at least one of these subgroups is of order at least $p^2$. Let $H$ be one with this property. Then, $|G/H|$ being at most $p^5$, it follows from \cite{KM05} that $\K(G/H) = \gamma_2(G/H)$. The proof  is now complete by Lemma \ref{cmlength} and assertion (iii).    \hfill $\Box$
 
 \end{proof}

\begin{lemma}\label{clsatlst5}
Let $G$ be a group  of order $p^7$,  $p \geq 5$,  such that $\gamma_2(G)$ is a $4$-generated abelian subgroup of order $p^5$.  Then the following statements hold:

(i) The nilpotency class of $G$ is at least  $5$,

(ii) If $p \ge 7$, then the nilpotency class of $G$ is exactly  $5$.
\end{lemma}
\begin{proof}
By the given hypothesis  we can assume that $G$ is generated by $2$ elements,  say $\alpha_1, \alpha_2$.  Hence it is clear that the nilpotency class of $G$ can neither be $2$ nor $3$.  Let it be $4$. 
Then, $\gamma_2(G)$ being abelian, its exponent  is $p^2$.  As we know  $\gamma_2(G)/\gamma_3(G) = \langle \beta \gamma_3(G) \rangle$.  And, as proved in the proof of Lemma \ref{cls5}, the order of $\beta$ in $G$ is $p^2$, $\exp(\gamma_3(G)) =p$ and $\Phi(\gamma_2(G)) = \gen{\beta^p} \le \Z(G)$ is of order $p$. Then $p \le |\gamma_3(G)/\gamma_4(G)| \le p^2$; hence $p^3 \ge |\gamma_4(G)| \ge p^2$. First assume that $|\gamma_4(G)| = p^3$. Let $K$ be a complement of $\Phi(\gamma_2(G))$ in $\gamma_4(G)$. Then $|K| = p^2$ and the quotient group $\bar{G}:= G/K$ is a $p$-group of maximal class and order $p^5$. Notice that the exponent  of $\gamma_2(\bar{G})$ is $p^2$, which is not possible by Theorem \ref{alc1}. 

Now assume that $|\gamma_4(G)| = p^2$. If $|\Z(G)| = p^3$, then $u := \beta_1^r \beta_2^s \in \Z(G)$ for some $r, s \in \mathbb{F}_p$ not simultaneously zero.  If $\beta^p \in \gamma_4(G)$, then, considering $H$ to be a complement  of $\gen{\beta^p}$ in $\gamma_4(G)$, it follows that $G/\gen{H, u}$ is a $p$-group of maximal class of order $p^5$ such that the exponent of its  commutator subgroup is $p^2$, which is not possible by Theorem \ref{alc1}. If $\beta^p \not\in \gamma_4(G)$, then we  choose an element $v \in \gamma_3(G) - \gamma_4(G)$ such that  $L:=\gen{\gamma_4(G), v}$ is a complement of $\gen{u}$ in $\gamma_3(G)$. Now the quotient group $G/L$ is a $p$-group of maximal class of order $p^4$ such that the exponent of its  commutator subgroup is $p^2$, which is again not possible by Theorem \ref{alc1}.

Finally assume that $|\Z(G)| = p^2$, that is, $\gamma_4(G) = \Z(G)$.   We claim that $\alpha_i^p \in \gamma_3(G)-\gamma_4(G)$ for  $i = 1,2$.   Contrarily assume that $\alpha_1^p \in \gamma_2(G) - \gamma_3(G)$.   Then $\alpha_1^p = \beta^t g$   for some $t \in \mathbb{F}_p^*$ and $g \in \gamma_3(G)$. Thus,  
$$1= [ \alpha_1^p, \alpha_1] = [\beta^t g, \alpha_1 ]= [\beta^t, \alpha_1][g, \alpha_1]= \beta_1^t[g, \alpha_1],$$
 which implies that $\beta_1 \in \gamma_4(G)$; not possible.  Now  assume that $\alpha_1^p \in \gamma_4(G)$.  Then, since $\exp(\gamma_3(G))=p$ and $p \ge 5$, we  obviously get 
  $$1=[\alpha_2, {\alpha_1}^p] = {[\alpha_2, \alpha_1]}^{p} = \beta^{-p} ,$$ which is not true.  The case for  $\alpha_2$ goes by symmetry, and the claim is settled. If ${\alpha_2}^p = (\alpha_1^p)^r$ for some $r \in \mathbb{F}_p^*$, then $$\beta^p=[\alpha_1, \alpha_2^{p}]= [\alpha_1, \alpha_1^{pr}]=1,$$   which is again not true. Hence $\alpha_1^p$ and $\alpha_2^p$ are both non-trivial and independent.

It is now clear that $\gamma_4(G) = \gen{[\alpha_1^p, \alpha_2]} \times \gen{[\alpha_1, \alpha_2^p]}$. If $\beta^p \in \gen{[\alpha_1^p, \alpha_2]}$, then consider $\bar{G} := G/\gen{\alpha_2^p, [\alpha_2^p, \alpha_1]}$. Otherwise consider $\bar{G} := G/\gen{\alpha_1^p, [\alpha_1^p, \alpha_2]}$. In both the cases $\bar{G}$ is a $p$-group of maximal class of order $p^5$ such that the exponent of its  commutator subgroup is $p^2$, which is again not possible by Theorem \ref{alc1}.  Hence the nilpotency class of $G$ is at least $5$.  This completes the proof of assertion (i).

Assertion (ii) is obvious from assertion (i) and Theorem \ref{alc1}, and the proof is complete.  \hfill $\Box$

\end{proof}

\begin{prop}\label{eabcls5}
Let $G$ be a group  of order $p^7$, $p \geq 5$, with nilpotency class $5$ and  $\gamma_2(G)$ elementary abelian of order $p^5$. Then the following statements hold:

(i) $|\gamma_5(G)|=p$,

(ii)  $|\Z(G/\gamma_5(G))| = p^2$ and  $|\Z_2(G)| = p^3$,

(iii) $\K(G) \ne \gamma_2(G)$, 

(iv) Commutator length of $G$ is $2$.
\end{prop}

\begin{proof}
By the given hypothesis  $G$ is minimally generated by $2$ elements,  say $\alpha_1, \alpha_2$, and $|\gamma_5(G)| \le p^2$.   Contrarily assume that $|\gamma_5(G)| = p^2$. Let $\beta$, $\beta_i$, $\eta_{ij}$ and $\xi_{ijk}$ are as defined in the introduction.  Without loss of generality we can assume that  $\gamma_3(G) = \gen{\beta_1,    \gamma_4(G)}$ and $\beta_2 \in \gamma_4(G)$.  Thus $\eta_{21}, \eta_{22} \in \gamma_5(G)$.  Now  Hall-Witt Identity
$$[\beta,  \alpha_1^{-1},  \alpha_2]^{\alpha_1} [\alpha_1,  \alpha_2^{-1},  \beta]^{\alpha_2} [\alpha_2,  \beta^{-1},  \alpha_1]^{\beta}= 1$$
gives $\eta_{12}= \eta_{21} \xi_{112} \xi_{121}^{-1}$.  Hence $\eta_{12} \in \gamma_5(G)$,  and therefore $\gamma_4(G)= \gen{\eta_{11}, \gamma_5(G)}$.  It is now clear that $\gamma_5(G)= \gen{\xi_{111},  \xi_{112}} $.   Again using Hall-Witt identity
\begin{center}
$[\beta_1,  {\alpha_2}^{-1},  \alpha_1]^{\alpha_2} [\alpha_2,  {\alpha_1}^{-1},  \beta_1]^{\alpha_1} [\alpha_1,  {\beta_1}^{-1},  \alpha_2]^{\beta_1}= 1,$
\end{center}
we get  $\xi_{112}=\xi_{121}$. Since  $\xi_{121} = 1$,  $\gamma_5(G)=\gen{\xi_{111}}$,  which is absurd. Hence $|\gamma_5(G)| = p$, and assertion (i) holds.

It now follows from \cite[Theorem 2.12]{NB58} that $|\gamma_4(G)/\gamma_5(G)| = p$, and therefore $|\gamma_3(G)/\gamma_4(G)| = p^2$.  By the given hypothesis,  it follows that $|\Z(G)| \le p^2$. If $|\Z(G)| = p^2$, then  $|\Z_2(G)| = p^3$ and $\Z(G/\gamma_5(G)) = (\gamma_4(G)/\gamma_5(G))(\Z(G)/\gamma_5(G))$ is of order $p^2$, and assertion (ii) holds in this case. So assume that $|\Z(G)| = p$. Since $\gamma_2(G)$ is elementary abelian,   for any $x,y \in G$,  we have
\begin{equation}\label{7}
[x,y^p]=[x,y]^p [x,y,y] ^{\binom p2} [x,y,y,y]^{\binom p3}[x,y,y,y,y]^{\binom p4}=1,
\end{equation}
which gives $G^p \le \Z(G)$.    Thus  $\bar{G} := G/\gamma_5(G)$ is of order $p^6$, exponent $p$ and nilpotency class  $4$  with $\gamma_2(\bar{G})$ elementary abelian of order $p^4$.    It then follows from  \cite{rJ80} that  $\bar{G}$ lies in one of the isoclinism classes $\Phi(23),  \Phi(40)$ and $\Phi(41)$.   We claim that $\bar{G}$  lies in $\Phi(23)$. 

If $\bar{G} \in \Phi(40)$, then, since $\exp(\bar G) = p$ and isoclinic groups have same commutator structure, we may assume that  
\begin{eqnarray*}
\bar{G} &=&  \langle \bar{\alpha_1},  \bar{\alpha_2},  \bar{\beta},  \bar{\beta_1},  \bar{\beta_2},  \bar{\gamma} \mid [\bar{\alpha_1},  \bar{\alpha_2}] = \bar{\beta},  [\bar{\beta},  \bar{\alpha_i}]= \bar{\beta_i}, [\bar{\beta_1}, \bar{\alpha_2}]= [\bar{\beta_2}, \bar{\alpha_1}]=\bar{\gamma}, \\
& & \;\  \bar{\alpha_i}^p=   \bar{\beta}^p =\bar{\beta_i}^p  = \bar{\gamma}^p=1; i=1,2 \rangle, 
\end{eqnarray*}
 Since $[\bar{\beta_i}, \bar{\alpha_i}]= 1_{\bar{G}}$,  we get $[\beta_i, \alpha_i] \in \gamma_5(G)$ for $i=1, 2$.  Hence $\gamma_4(G)=\gen{ [\beta_1, \alpha_2], \gamma_5(G)}$  and $[\beta_2,  \alpha_1] = [\beta_1,  \alpha_2] h$,  for some $h \in \Z(G)$.     By Hall-Witt Identity 
\begin{center}
$[\beta_1,  {\alpha_2}^{-1},  \alpha_1]^{\alpha_2} [\alpha_2,  {\alpha_1}^{-1},  \beta_1]^{\alpha_1} [\alpha_1,  {\beta_1}^{-1},  \alpha_2]^{\beta_1}= 1,$
\end{center}
we get $[\beta_1,  \alpha_2,  \alpha_1]=1$.  Again using  
\begin{center}
$[\beta_2,  {\alpha_1}^{-1},  \alpha_2]^{\alpha_1} [\alpha_1,  {\alpha_2}^{-1},  \beta_2]^{\alpha_2} [\alpha_2,  {\beta_2}^{-1},  \alpha_1]^{\beta_2}= 1,$
\end{center}
we get $[\beta_2,  \alpha_1,  \alpha_2]=1$.  Hence $G = \C_G(\gamma_4(G))$,  which is absurd. 

If $\bar{G} \in \Phi(41)$, then using the presentation
\begin{eqnarray*}
\bar{G} &=&  \langle \bar{\alpha}_1,  \bar{\alpha}_2,  \bar{\beta},  \bar{\beta}_1,  \bar{\beta}_2,  \bar{\gamma} \mid [\bar{\alpha}_1,  \bar{\alpha}_2] = \bar{\beta},  [\bar{\beta},  \bar{\alpha}_i]= \bar{\beta}_i, [\bar{\alpha}_1, \bar{\beta}_1]^{-\nu}= [\bar{\alpha}_2, \bar{\beta}_2]=\bar{\gamma}^{-\nu}, \\
& & \;\  \bar{\alpha}_i^p=   \bar{\beta}^p =\bar{\beta}_i^p  =\bar{\gamma}^p=1; i=1,2 \rangle, 
\end{eqnarray*}
one gets the same absurd conclusion as in the preceding case.  Hence $\bar{G} \in \Phi(23)$ and therefore  $|Z(\bar{G})|=p^2$ by \cite{rJ80}, which proves assertion (ii).  

 Assertion (iii) now follows from assertion (ii), Lemma \ref{p6} and Lemma \ref{quotient}. The proof of assertion (iv) is exactly same as the  proof of assertion (iv) of Lemma \ref{cls5} for abelian $\gamma_2(G)$. The proof is complete.  \hfill  $\Box$

\end{proof}

\begin{prop}\label{nab}
Let $G$ be a group of order $p^7$ and nilpotency class $5$ with $\gamma_2(G)$  non-abelian of order $p^5$ and  exponent $p$.   Then $\K(G) \neq \gamma_2(G)$ if and only if $|\gamma_5(G)| = p$  and  $|\Z_2(G)| = p^3$. Moreover, the commutator length of $G$ is at most $2$.
\end{prop}
\begin{proof}
By the given hypothesis it follows that $|\gamma_5(G)| \le |\Z(G)| \le p^2$ and the commutator subgroup of $\gamma_2(G)$ is contained in $\gamma_5(G)$. First assume that $|\gamma_5(G)| = p$ and  $|\Z_2(G)| = p^3$. It then follows that $|\Z(G/\gamma_5(G))| = p^2$, and therefore  $ \K(G) \neq \gamma_2(G)$ by Lemma  \ref{p6} and  Lemma \ref{quotient}.   Conversely, assume that either $|\gamma_5(G)| = p^2$ or $|\gamma_5(G)| = |\Z(G)| = p$ and  $|\Z_2(G)| = p^2$.  Notice that $|\Z(G/\gamma_5(G))|=p$ in the second situation.
Let $G = \gen{\alpha_1, \alpha_2}$ and $\beta$,  $\beta_i$,  $\eta_{ij}$ and $\xi_{ijk}$,  where $i,j,k \in \{1,2\}$, be as defined in the introduction.

 Let us first assume that $|\gamma_5(G)| = |\Z(G)| = p^2$.    Notice that $|\gamma_i(G)/\gamma_{i+1}(G)| = p$ for $2 \le i \le 4$.  Without loss of generality we can assume that  $\beta_2 \in \gamma_4(G)$ and $\gamma_3(G)= \gen{\beta_1, \gamma_4(G)}$.   Therefore $\eta_{21},  \eta_{22} \in \gamma_5(G)$.  But by Hall-Witt identity 
\begin{center}
$[\beta,  \alpha_1^{-1},  \alpha_2]^{\alpha_1} [\alpha_1,  \alpha_2^{-1},  \beta]^{\alpha_2} [\alpha_2,  \beta^{-1},  \alpha_1]^{\beta}= 1,$
\end{center} 
we get  $\eta_{12}=  \xi_{112} {\xi_{121}}^{-1} {\eta_{21}}$;  thus $\eta_{12} \in \gamma_5(G)$.  Hence $\gamma_4(G)= \gen{\eta_{11}} \pmod {\gamma_5(G)}$ and therefore $\gamma_5(G)=\gen{ \xi_{111}, \xi_{112}}$.  Again the identity
\begin{center}
$[\beta_1,  {\alpha_2}^{-1},  \alpha_1]^{\alpha_2} [\alpha_2,  {\alpha_1}^{-1},  \beta_1]^{\alpha_1} [\alpha_1,  {\beta_1}^{-1},  \alpha_2]^{\beta_1}= 1,$
\end{center}
 gives $[\beta_1,  \beta] = \xi_{112}$. 
Let $\beta_2= {\eta_{11}}^r {\xi_{111}}^s {\xi_{112}}^t$ for some $r,s,t \in \mathbb{F}_p$.  Let $H=\gen{\xi_{112}}$ and $\bar{G} := G/H$.  Then $\bar{G}$ is a $p$-group of maximal class of order $p^6$ with  $\gamma_2({\bar{G}})$  elementary abelian of order $p^4$. Notice that $\bar{\alpha}_1$ is an uniform element of $\bar G$. It is now easy to see that
 $$\gamma_2(\bar{G}) = \{[ \bar{\alpha}_2^a \bar{\beta}^b \bar{\beta}_1^c \bar{\eta}_{11}^d , \bar{\alpha}_1] \mid a,b,c,d \in \mathbb{F}_p\} = \{[\bar{\alpha}_2^a \bar{\beta}^b \bar{\beta}_1^c \bar{\eta}_{11}^d, \bar G] \mid a,b,c,d \in \mathbb{F}_p\}.$$
   Since $\xi_{112}=[\beta_1, \beta]= [\eta_{11}, \alpha_2]$,   it also follows that 
 $$ H \subseteq   \bigcap \limits_{\substack{a , b,  c,  d }}[ \alpha_2^a \beta^b \beta_1^c \eta_{11}^d , G],$$ where $a, b, c, d \in \mathbb{F}_p$ not all simultaneously zero. Hence by Lemma \ref{prelemma2} we have $\K(G)=\gamma_2(G)$.

Now assume that $|\gamma_5(G)| = |\Z(G)| = p$ and  $|\Z(G/\gamma_5(G))|=p$.    Let $ H := Z(G)$ $(= \gamma_5(G))$ and $L := G/H$.  Then $L$ is of order $p^6$ and nilpotency class $4$ such that  $\gamma_2(L)$ is elementary abelian of order $p^4$ and $|\Z(L)|=p$. Hence, by \cite{rJ80}, $L$ lies in  the isoclinism class $\Phi(40)$ or  $\Phi(41)$.  By  \cite[Lemma 4.1]{MR}  we know that $\K(L)= \gamma_2(L)$.  We provide a detailed explanation when $L$ lies in $\Phi_{40}$. The other case goes on the same lines, and therefore is left for the reader.  Notice that any two isoclinic groups admit the same commutator relations. So, if  $L$ lies in $\Phi(40)$, then by  \cite{rJ80} $L$ has the following commutator relations:
$$R := \{[ \bar{\alpha}_1, \bar{\alpha}_2]= \bar{\beta},  [\bar{\beta}, \bar{\alpha}_i]= \bar{\beta}_i, [\bar{\beta}_1, \bar{\alpha}_2]=[\bar{ \beta}_2, \bar{\alpha}_1]= \bar{\eta}\},$$
where $\gamma_4(G)/\gamma_5(G) = \gen{\eta \gamma_5(G)}$. It is  not difficult to see that $\bar{\alpha}_i^p = 1$ for $i = 1, 2$.

We claim that $\gamma_2(L)=\{ [{\bar{\alpha}_1}^{\epsilon} \bar{\beta}^{a_1} {\bar{\beta}_1}^{a_2} {\bar{\beta}_2}^{a_3},  L] \mid a_1, a_2 , a_3 \in \mathbb{F}_p ,   \epsilon \in \{0,1\} \}$,  where $a_1$ and $a_2$ are not simultaneously zero. To establish this claim, we are going to show that for given  $u,v,w,x \in \mathbb{F}_p$,  there exist $\epsilon,  a_i, b_j \in \mathbb{F}_p$,  where $1 \leq i \leq 3$, $1 \leq j \leq 5$ and   $a_1,  a_2$  not simultaneously zero, such that 
$$[{\bar{\alpha}_1}^{\epsilon} \bar{\beta}^{a_1} {\bar{\beta}_1}^{a_2} {\bar{\beta}_2}^{a_3},   {\bar{\alpha}_1}^{b_1} {\bar{\alpha}_2}^{b_2} \bar{\beta}^{b_3} {\bar{\beta}_1}^{b_4} \bar{\beta}_2^{b_5} ]  =  \bar{\beta}^{u}\bar{\beta}_1^{v}\bar{\beta}_2^{w} \bar{\eta}^{x}.$$
After expanding the left hand side and comparing the powers on both sides, we get 
\begin{eqnarray*}
\epsilon b_2 =u,\\
a_1 b_1 - \epsilon b_3     =v,\\
a_1 b_2  + \epsilon  {\binom {b_2} 2}    =w,\\
a_1 b_1 b_2 + a_2 b_2 + a_3 b_1 -\epsilon  b_5  =x.
\end{eqnarray*}

If $u \neq 0$, then taking $\epsilon =1$, $a_2 \in \mathbb{F}_p^*$ and $b_1=0$,  we get $b_2 =u$,  $ b_3    =-v$, $a_1 = (w- {\binom{u}2})u^{-1}$,   and $b_5=  ua_2  - x$.  Remaining $a_i$'s and $b_i$'s can take arbitrary values in $\mathbb{F}_p$.  Now let $u=0$.  Then taking $\epsilon =0$, the above system of equations reduces to
\begin{eqnarray}
a_1 b_1      =v, \label{30} \\
   a_1 b_2     =w, \label{31} \\
a_1 b_1 b_2 + a_2 b_2 + a_3 b_1   =x. \label{32}
\end{eqnarray}
If $v=w=0$,  then taking $a_1=a_3=0$ and $a_2 \neq 0$,  we get  $ b_2    =x a_2^{-1}$.  If $v\neq 0$ and $w \neq 0$,  then we can choose $b_1, b_2 \in \mathbb{F}_p^*$ such that   $a_1= v b_1^{-1}= w b_2^{-1} $.  Now taking $a_3=0$ in \eqref{32},  we get $  a_2     =b_2^{-1}(x-a_1 b_1 b_2) $.  If $v=0$ and $w \neq 0$,  then  by equations \eqref{31}  and \eqref{30}, respectively, we get  $a_1= w b_2^{-1}$ and $b_1=0 $,     and by \eqref{32} we get $a_2 = xb_2^{-1}$.  If $w=0$ and $v \neq 0$,  then  by equations \eqref{30} and \eqref{31}, respectively,  we get $a_1= v b_1^{-1}$ and $b_2=0 $,     and by \eqref{32} we get $a_3 = xb_1^{-1}$.  Thus in all the cases we get the required $\epsilon$, $a_i$'s and $b_i$'s, and therefore it follows that
 $$ \gamma_2(L) =   \bigcup \limits_{\substack{\epsilon , a_1,  a_2,  a_3 }}[ \alpha_1^{\epsilon} \beta^{a_1} \beta_1^{a_2} \beta_2^{a_3},  L],$$
 where $a_1$ and $a_2$ are not simultaneously zero.

As $\gamma_2(G)$ is non-abelian,   we can assume that $[\beta, \beta_1] \neq 1$ and $[\beta, \beta_2]=1$.  For, if both are non-zero, then $[\beta, \beta_1]= [\beta, \beta_2]^t$ for some $t \in \mathbb{F}_p^*$.  Then modifying $\beta_2$ by $\beta_1 \beta_2^{-t}$,  we get $[\beta, \beta_2]=1$. If  $[\beta, \beta_1] =1$, then interchanging the role of $\alpha_1$ and $\alpha_2$ works.   By Hall-Witt Identity 
\begin{center}
$[\beta_1,  {\alpha_1}^{-1},  \alpha_2]^{\alpha_1} [\alpha_1,  {\alpha_2}^{-1},  \beta_1]^{\alpha_2} [\alpha_2,  {\beta_1}^{-1},  \alpha_1]^{\beta_1}= 1,$
\end{center}
we get  $\xi_{121}= [\beta, \beta_1]$.  Again by Hall-Witt Identity
\begin{center}
$[\beta_2,  {\alpha_1}^{-1},  \alpha_2]^{\alpha_1} [\alpha_1,  {\alpha_2}^{-1},  \beta_2]^{\alpha_2} [\alpha_2,  {\beta_2}^{-1},  \alpha_1]^{\beta_2}= 1,$
\end{center}
we get  $\xi_{212}=1$. Thus $\alpha_2 \in C_G(\gamma_4(G))$.  Therefore we get $\gamma_5(G)= \gen{\xi_{121}}$.   Now, since $\gamma_5(G) =\gen{[\alpha_1, \gamma_4(G) ]}= \gen{[\beta, \beta_1]}$, it follows that    
$$ \gamma_5(G) \subseteq   \bigcap \limits_{\substack{\epsilon , a_1,  a_2,  a_3 }}[ \alpha_1^{\epsilon} \beta^{a_1} \beta_1^{a_2} \beta_2^{a_3}, G],$$
 where  $a_1$ and $a_2$ are not both zero.  Hence, again by Lemma \ref{prelemma2}, $\K(G)=\gamma_2(G)$.
 
 If $L$ lies in $\Phi(41)$, then, on the same lines, one can show that
$$\gamma_2(G) = \{ [\alpha_2^{\epsilon} \beta^{a_1} \beta_1^{a_2} \beta_2^{a_3},  G] \mid a_1,  a_2, a_3 \in \mathbb{F}_p ,   \epsilon \in \{0,1\} \}.$$ 
That the commutator length of $G$ is at most $2$ follows the same way as proved in  Lemma \ref{cls5}. The proof is now complete.     \hfill $\Box$

\end{proof}

%%%%%%%%%%%%%%%%%%%%%%%%%%%%%%%%%%%%%%%%%%%%%%%%%
%%%%%%%%%%%%%%%%%%%%%%%%%%%%%%%%%%%%%%%%%%%%%%%%%

\section{Groups of nilpotency class $4$}	

Let $G$ be a group of nilpotency class $4$ and order $p^7$ with $\gamma_2(G)$  of order $p^5$ and minimally generated by at least $4$ elements. By Lemma \ref{nab1} we know that $\gamma_2(G)$ is abelian. It now follows from Lemma \ref{clsatlst5} that $\exp(\gamma_2(G))$ for such groups $G$ can not be $p^2$. So the only possibility is  elementary abelian $\gamma_2(G)$, which we'll deal in this section.

\begin{lemma}\label{cls4}
Let $G$ be a group of order $p^7$ and   nilpotency class $4$ with $\gamma_2(G)$  elementary abelian of order $p^5$. Then $\Z(G) = \gamma_4(G)$ is of order $p^2$.
\end{lemma}
\begin{proof}
We can assume $G=\gen{\alpha_1, \alpha_2}$.  Let $\beta$, $\beta_i$ and  $\eta_{ij}$ are as defined in  the introduction.   Then   $\gamma_2(G)/\gamma_3(G)=\gen{\beta \gamma_3(G)}$. Obviously  $|\gamma_4(G)| \le p^3$. If $|\gamma_4(G)| = p^3$, then $\gamma_3(G) = \gen{\gamma, \gamma_4(G)}$ for any $\gamma \in \gamma_3(G) - \gamma_4(G)$. This implies that $\gamma_4(G) =\gen{[\gamma, \alpha_1], [\gamma, \alpha_2]}$, which is not possible. Hence the order of $\gamma_4(G)$ is at most  $p^2$. It is easy to see that it is precisely $p^2$,  and therefore  $\gamma_3(G)/\gamma_4(G)=\gen{\beta_1 \gamma_4(G),  \beta_2\gamma_4(G)}$ is of order $p^2$.    Using Hall-Witt identity  
\begin{center}
$[\beta,  \alpha_1^{-1}, \alpha_2]^{\alpha_1}[\alpha_1, \alpha_2^{-1}, \beta]^{\alpha_2}[\alpha_2, \beta^{-1}, \alpha_1]^{\beta}=1,$
\end{center} 
we have $\eta_{12}= \eta_{21}$. 

Contrarily assume that $|\Z(G)| = p^3$.   Since $|\gamma_4(G)| = p^2$,  we have  $Z(G)= \gen{\beta_1^i \beta_2^j,  \gamma_4(G) }$ for some $i, j \in \mathbb{F}_p$  not both zero.  If $i=0$,  then $\eta_{22}= \eta_{21}=1$,  which implies that $\gamma_4(G) =\gen{\eta_{11}}$; not possible.  So let  $i \in \mathbb{F}_p^*$.  As $1= [\beta_1^i \beta_2^j,  \alpha_1]= \eta_{11}^i \eta_{21}^j$ and $1=[\beta_1^i \beta_2^j, \alpha_2]=  \eta_{12}^i \eta_{22}^j$,  we get $\eta_{11}= {\eta_{12}}^{-ji^{-1}} =  {\eta_{22}}^{j^2i^{-2}}$,  which implies that the order $\gamma_4(G)$ is at most $p$; not possible. The proof is now complete. \hfill $\Box$

\end{proof}

Let $G$ satisfy the hypotheses of Lemma \ref{cls4}. Then $\Z(G) = \gamma_4(G)$ is of order $p^2$. Assume that $\gamma_4(G)=\gen{\eta_{11}, \eta_{12}}$ and $\eta_{22} =   \eta_{11}^m \eta_{12}^n$,  where $m, n \in \mathbb{F}_p$ not both zero. With this setting we have

\begin{lemma}\label{sbgrpH}
Let $G$, $m$ and $n$ be as in the preceding para. Then $G$ admits a subgroup $H \leq Z(G)$  of order $p$ such that $|Z(G/H)|=p^2$ if and only if  the equation $m \lambda^2 + n \lambda \mu - \mu^2=0$ has a solution $\lambda = \lambda_0$ ($\neq 0$) and $\mu = \mu_0$ in $\mathbb{F}_p$. 
\end{lemma}
\begin{proof}
First assume that $G$ admits a subgroup $H \leq Z(G)$  of order $p$ such $|Z(G/H)|=p^2$. Obviously $H := \gen{\eta_{11}^a \eta_{12}^b}$ for some $a, b \in \mathbb{F}_p$ not simultaneously zero. Then there exist $i, j \in \mathbb{F}_p$ not both zero such that $[\beta_1^i \beta_2^j, \alpha_k] \in H$ for $k = 1, 2$; meaning, there exist $\lambda, \mu \in \mathbb{F}_p$ such that 
$$\eta_{11}^i \eta_{12}^j = (\eta_{11}^a \eta_{12}^b)^{\lambda}$$
 and 
 $$\eta_{11}^{mj} \eta_{12}^{i+nj} = (\eta_{11}^a \eta_{12}^b)^{\mu}.$$
  Obviously $\lambda \neq 0$. Comparing powers on both sides in these equations, we get
\begin{eqnarray*}
i-a \lambda= 0 \pmod p, \ \ j- b \lambda = 0\pmod p,\\
i+nj-b\mu = 0 \pmod p, \ \ mj - a\mu = 0 \pmod  p.
\end{eqnarray*}
Solving these equations for $i,j$ we further get
\begin{eqnarray*}
\lambda a+(n\lambda - \mu)b=0\pmod p\\
-\mu a+m\lambda b=0\pmod p.
\end{eqnarray*} 
Hence, the system of equations
\begin{eqnarray*}
\lambda x+(n\lambda - \mu)y=0\pmod p\\
-\mu x+m\lambda y=0\pmod p.
\end{eqnarray*}
admits a non-trivial solution $x = a$ and $y = b$ in $\mathbb{F}_p$. But it is possible only when the determinant $m\lambda ^2+n\lambda \mu - \mu^2 $ of the matrix of this system of equations is zero. Hence the equation $m\lambda ^2+n\lambda \mu - \mu^2 = 0$, with given coefficients $m$ and $n$,   admits a non-trivial solution $\lambda = \lambda_0 \in \mathbb{F}_p^*$ and $\mu = \mu_0$ in $\mathbb{F}_p$. 

Conversely, assume that $m \lambda^2 + n \lambda \mu - \mu^2=0$ admits a solutions $\lambda = \lambda_0$ ($\neq 0$) and $\mu = \mu_0$ in $\mathbb{F}_p$. Consider the matrix
\[
M=
  \begin{bmatrix}
    \lambda_0 & n\lambda_0 - \mu_0   \\
    -\mu_0 & m \lambda_0   \\
  \end{bmatrix}.
\]
Since the determinant of $M$ is  zero  modulo $p$,  the system of equations  $MA^t= \bf{0} $ admits a non-trivial solution, where $A^t$ denotes the transpose of the matrix $A=[a,b]$ and $\bf{0}$ denotes the $2 \times 1$ zero-matrix with entries in $\mathbb{F}_p$.   Let $a = a_0$ and $b = b_0$ be a non trivial solution of $MA^t= \bf{0} $.  Thus we get the following equations: 
\begin{eqnarray*}
\lambda_0 a_0+(n\lambda_0 - \mu_0)b_0=0\pmod p, \\
-\mu_0 a_0+m\lambda_0 b_0=0\pmod p.
\end{eqnarray*}

Let $H :=\gen{\eta_{11}^{a_0}  \eta_{12}^{b_0}}$ and  $\bar{G} := G/H$.     Obviously $[ \bar{\beta_1}^{a_0 } \bar{\beta_2}^{b_0 },  \bar{\alpha_1} ]= \bar{\eta}_{11}^{a_0}  \bar{\eta}_{12}^{b_0} = 1_{\bar{G}} $.  Also, using the preceding set of equations, we get
$$[ \bar{\beta_1}^{a_0 } \bar{\beta_2}^{b_0 },  \bar{\alpha}_2 ]= \bar{\eta}_{21}^{a_0}  \bar{\eta}_{22}^{b_0} =  \bar{\eta}_{11}^{m b_0} \bar{ \eta}_{12}^{a_0+nb_0}= \bar{\eta}_{11}^{ a_0 \mu_0 \lambda_0^{-1}}  \bar{\eta}_{12}^{  b_0 \mu_0 \lambda_0^{-1}} =  (\bar{\eta}_{11}^{a_0}  \bar{\eta}_{12}^{b_0})^{\mu_0 \lambda_0^{-1} } = 1_{\bar{G}}.$$
Hence $ \bar{\beta_1}^{a_0 } \bar{\beta_2}^{b_0 } \in Z(\bar{G})$. It  now easily follows that   $|\Z(\bar{G})| = p^2$, which completes the proof.    \hfill $\Box$

\end{proof}

\begin{thm}\label{eabcls4}
Let $G$ be a group of order $p^7$ and nilpotency class $4$ with $\gamma_2(G)$  elementary abelian of order $p^5$.  Then $\K(G) \neq \gamma_2(G)$ if and only if there exist a subgroup $H \leq Z(G)$  of order $p$ such $|Z(G/H)|=p^2$. Moreover, the commutator length of $G$ is at most $2$.
\end{thm}
\begin{proof}
 Let $H \leq Z(G)$ be a subgroup of order $p$ such that $|Z(G/H)|=p^2$. Then it follows from  Lemma \ref{p6} and Lemma \ref{quotient} that $ \K(G) \neq \gamma_2(G)$, which proves the `if' part. We provide a contrapositive prove of the `only if'   statement. Let there do not exist any  $H \leq Z(G)$ of order $p$ such $|Z(G/H)|=p^2$.  We assume $G=\gen{\alpha_1, \alpha_2}$.  Let $\beta$, $\beta_i$ and  $\eta_{ij}$ are as defined in  the introduction.   Then   $\gamma_2(G)/\gamma_3(G) = \gen{\beta \gamma_3(G)}$.  By Lemma \ref{cls4} we know   that $\Z(G) = \gamma_4(G)$ is of order $p^2$; hence  $\gamma_3(G)/\gamma_4(G)=\gen{\beta_1 \gamma_4(G),  \beta_2 \gamma_4(G)}$ is of order $p^2$ too. Using Hall-Witt identity
\begin{center}
$[\beta,  \alpha_1^{-1}, \alpha_2]^{\alpha_1}[\alpha_1, \alpha_2^{-1}, \beta]^{\alpha_2}[\alpha_2, \beta^{-1}, \alpha_1]^{\beta}=1,$
\end{center} 
 we get $\eta_{12}= \eta_{21}$. It  now follows from our assumption  that $\eta_{11}, \eta_{12}, \eta_{22}$ are all non-trivial elements of $\gamma_4(G)$. Indeed, for example, if $\eta_{11} =1$, then taking $H = \gen{\eta_{12}}$ we see that $|\Z(G/H)| = p^2$, a case which we are not considering.  Assume that $\gamma_4(G)=\gen{\eta_{11}, \eta_{22}}$.  Hence $\eta_{12}= \eta_{11}^r \eta_{22}^s$,  where $r, s \in \mathbb{F}_p$.  If $r=0$,  then taking  $H=\gen{\eta_{22}}$,  we get $\Z(\bar G)= \gen{\bar{\eta}_{11}, \bar{\beta}_2}$ is of order $p^2$; not possible.   If $s=0$,  then taking $H=\gen{\eta_{11}}$, we get  the same conclusion.    So let $r, s \in \mathbb{F}_p^*$.  Then we can take $\gamma_4(G)=\gen{\eta_{11}, \eta_{12}}$, and $\eta_{22}= \eta_{11}^m \eta_{12}^n$,  where $m,n \in \mathbb{F}_p$. The same conclusion holds when we take   $\gamma_4(G)=\gen{\eta_{12}, \eta_{22}}$.  Thus we can always consider   $\gamma_4(G)=\gen{\eta_{11}, \eta_{12}}$, and $\eta_{22}= \eta_{11}^m \eta_{12}^n$ for some $m, n \in \mathbb{F}_p$ not both zero.

Notice that $G$ satisfies all hypotheses of Lemma  \ref{sbgrpH}. Hence, under our supposition, the equation $m \lambda ^2+n\lambda \mu - \mu^2 = 0$  can only have the trivial solution  modulo $p$.  But, by Lemma \ref{quad}, it is possible if and only if the equation $ \lambda ^2+ n\lambda \mu - m \mu^2 = 0$   also  has  only the trivial solution  modulo $p$. Then, obviously,  $m \in \mathbb{F}_p^*$.  

We are going to prove that for any given choice of elements $u,v,w,x,y \in \mathbb{F}_p$,  there exist elements $a_i, b_i \in \mathbb{F}_p$, $1 \le i \le 5$,  such that 
\begin{center}
$[{\alpha_1}^{a_1} {\alpha_2}^{a_2} {\beta}^{a_3} {\beta_1}^{a_4} {\beta_2}^{a_5},  {\alpha_1}^{b_1} {\alpha_2}^{b_2} {\beta}^{b_3} {\beta_1}^{b_4} {\beta_2}^{b_5}]={\beta}^{u}{\beta_1}^{v}{\beta_2}^{w} {\eta_{11}}^{x} {\eta_{12}}^{y}.$
\end{center}
Solving both sides and comparing the powers in the preceding equation we get
\begin{eqnarray*}
a_1 b_2 -a_2 b_1 &=&u,\\
a_3 b_1 -a_1 b_3 + {\binom {a_1} 2} b_2 - a_2{\binom{b_1}2}   &=&v,\\
a_3 b_2 -a_2 b_3 + a_1{\binom {b_2} 2}  - {\binom{a_2}2}b_1 + a_1 a_2 b_2 - a_2 b_1 b_2  &=&w,\\
a_4 b_1 -a_1 b_4 + a_3{\binom {b_1} 2}  - {\binom{a_1}2}b_3 +{\binom {a_1} 3} b_2 - a_2{\binom{b_1}3}  + mA&=&x,\\
a_5 b_1 +a_4 b_2 + a_3 b_1 b_2 - a_1 b_5 -a_2 b_4 - a_1 a_2 b_3  +{\binom {a_2} 2} {\binom{b_1}2} && \\ 
- {\binom{a_1}2}{\binom{b_2}2} + a_2 b_2{\binom {a_1} 2} - a_2 b_2 {\binom{b_1}2}  +nA&=&y,
\end{eqnarray*}
 where $A= a_5b_2-a_2b_5 + a_3 \binom{b_2}2 - \binom{a_2}2 b_3 + a_1 \binom{b_2}3 - \binom{a_2}3 b_1 + a_1 \binom{a_2}2 b_2 - a_2 b_1 \binom{b_2}2 + a_1 a_2 \binom{b_2}2 - \binom{a_2}2 b_1 b_2$. We shall call this system of equations \emph{the original system of equations} (OSE) throughout the proof.

\textit{Case (i)}.  $u \neq 0$. Taking $a_2= a_5= b_1 = b_5= 0 $ and $a_1= 1$ in OSE, we get
\begin{eqnarray*}
 b_2  =u,  \ \ b_3    =-v, \ \  a_3 b_2  + {\binom {b_2} 2}   =w,\\
 - b_4   + m( a_3 \binom{b_2}2  +  \binom{b_2}3 )=x,\\
a_4 b_2   +n( a_3 \binom{b_2}2  +  \binom{b_2}3 )=y.
\end{eqnarray*}
Hence we get 
\begin{eqnarray*}
 b_2  =u,  \ \ b_3    =-v, \ \  a_3 =u^{-1}(w-  {\binom {u} 2}),\\
  b_4  = m( a_3 \binom{u}2  +  \binom{u}3 )-x,\ \ a_4 = u^{-1}\big( y-n( a_3 \binom{u}2  +  \binom{u}3 )\big),
\end{eqnarray*}
which are the required values of $a_i$'s and $b_i$'s.

\textit{Case (ii)}.   $u=0$.  Let us fix $a_1=a_2=0$. Then OSE reduces to 
\begin{eqnarray}
a_3 b_1    =v, \ \ a_3 b_2  =w,  \label{12}\\
a_4 b_1  + a_3{\binom {b_1} 2}  + m(a_5b_2 + a_3 \binom{b_2}2)=x, \label{13}\\
a_5 b_1 +a_4 b_2 + a_3 b_1 b_2 +n(a_5b_2 + a_3 \binom{b_2}2)=y.  \label{14}
\end{eqnarray}
 We now consider the following four possible subcases:
 
\textit{Subcase  (i)}.  $v=w=0$.  Taking $a_3=b_2=0$ and $b_1=1$ in \eqref{13} and \eqref{14},   we get  $a_4  = x$ and $a_5 =y$.  Notice that $b_j$, $3 \le j \le 5$, can take arbitrary values in $\mathbb{F}_p$.

\textit{Subcase (ii)}.  $v=0$,  $w \neq 0$.  Taking $b_1=0$ and $b_2=1$ in \eqref{12}, \eqref{13} and \eqref{14},  we get
$a_3= w$,   $a_5= xm^{-1} $ and $a_4=y-n a_5$. Again $b_j$, $3 \le j \le 5$, can take arbitrary values in $\mathbb{F}_p$.

\textit{Subcase (iii)}.  $v\neq 0$,  $w = 0$.  Taking $b_1 = 1$ and $b_2=0$ in \eqref{12}, \eqref{13} and \eqref{14},  we get   $a_3 = v$,  $a_4 = x $ and $a_5 = y$. Further again  $b_j$, $3 \le j \le 5$, can take arbitrary values in $\mathbb{F}_p$.

\textit{Subcase (iv)}.  $v\neq 0$, $w \neq 0$.  For this choice, we can choose  $b_1, b_2 \in \mathbb{F}_p^*$ such that  $a_3 := v \, b_1^{-1}= w \, b_2^{-1}$.
Rewriting \eqref{13} and \eqref{14},  we get 
\begin{eqnarray}
 b_1 a_4   + mb_2 a_5 = x-a_3{\binom {b_1} 2} - m  a_3 \binom{b_2}2, \label{15}\\
  b_2  a_4 + (b_1+ nb_2)a_5      = y-a_3 b_1 b_2- n a_3  \binom{b_2}2.\label{16}
\end{eqnarray}
Viewing $a_4$ and $a_5$ as variables, the determinant of the matrix of this system of equations is  
 $D := b_1^2+nb_1b_2 - mb_2^2$.  As explained above, in the second para of this proof, $D$ can not be zero. Hence we can solve the preceding system of equations to obtain $a_4$ and $a_5$. Taking $b_j$, $3 \le j \le 5$, any arbitrary elements of $\mathbb{F}_p$, we got the required $a_i$'s and $b_i$'s, and the proof of the main assertion is complete.   
 
 As we know $\beta_1 = [\beta, \alpha_1]$ and  $\beta_2 = [\beta, \alpha_2]$ are both non-trivial. Using the fact that $\eta_{12} = \eta_{21}$ and $|\gamma_4(G)| = p^2$, it follows that $\alpha_1$ can not centralise $\beta_1$ and $\beta_2$ both. Then $H := [\alpha_1, \gamma_2(G)] \subseteq \K(G)$ is a normal subgroup of $G$ having order at least $p^2$. For, $\gamma_2(G)$ being abelian, we have
$[\alpha_1, uv] = [\alpha_1, u] [\alpha_1, v]$ for all $u, v \in \gamma_2(G)$.
Since $|G/H| \le p^5$, it follows from \cite{KM05} that $\K(G/H) = \gamma_2(G/H)$. The proof is now complete by Lemma \ref{cmlength}.     \hfill $\Box$

\end{proof}

%%%%%%%%%%%%%%%%%%%%%%%%%%%%%%%%%%%%%%%%%%%%%%%%%%%%
%%%%%%%%%%%%%%%%%%%%%%%%%%%%%%%%%%%%%%%%%%%%%%%%%%%%

\section{Proof of Theorem A and Examples}

So far we have handled groups of order $p^7$ with $|\gamma_2(G)| = p^5$ and $d(\gamma_2(G)) \ge 4$. The situation when $d(\gamma_2(G)) \le 3$ has already been taken care of in the literature. So we are only left with the case when $|\gamma_2(G)| = p^4$ and $d(\gamma_2(G)) = 4$, which we'll consider now.

 \begin{lemma}\label{lastlem}
Let $G$ be a group of order $p^7$ with $|\gamma_2(G)| = p^4$ and $d(\gamma_2(G)) = 4$. Then $\K(G) \neq \gamma_2(G)$ if and only if one of the following holds:

(i) The nilpotency class of $G$ is $3$, $Z(G) \le \gamma_2(G)$  and   $|\Z(G)| = p^3$,

(ii) The nilpotency class of $G$ is $4$, $Z(G) \not\le \gamma_2(G)$ and   $|\Z(G)| = p^3$.

Moreover, the commutator length of $G$ is at most $2$.
\end{lemma}
 \begin{proof} 
 By the given hypothesis  the exponent of  $\gamma_2(G)$ is $p$.   If $\Z(G) \le \gamma_2(G)$, then the assertion follows from \cite[Theorem A]{MR}. So assume that $Z(G) \not\le \gamma_2(G)$. Thus $|\Z(G)| \ge p^2$. Notice that, in this case,  the nilpotency class of $G$ is at least $4$. Obviously, $d(G)$ is either $2$ or $3$. If $d(G) = 2$, then there exists a minimal generating set $\{\alpha_1, \alpha_2\}$ such that $\alpha_2^p \in \Z(G) - \gamma_2(G)$, but $\alpha_2^{p^2} \in \gamma_2(G)$. Then, by Lemma \ref{hlemma}, there exists a $2$-generated group $M$ of order $p^6$ such that $M$ and $G$ are isoclinic and $\Z(M) \le \gamma_2(M)$.  Also $|\Z(G)| = p |\Z(M)|$.  It now  follows from  \cite[Theorem A]{MR} that $\K(M) \neq \gamma_2(M)$ if and only if the nilpotency class of $M$ is $4$ and  $\Z(M)$ is of order $p^2$. This, using Lemma \ref{prelemma},  simply tells that $\K(G) \neq \gamma_2(G)$ if and only if the nilpotency class of $G$ is $4$ and  $|\Z(G)| = p^3$.

We now assume that $d(G) = 3$. We can then choose a minimal generating set  $\{\alpha_1, \alpha_2, \alpha_3\}$ for $G$ such that $\alpha_3 \in \Z(G) - \gamma_2(G)$. By the given hypothesis $\alpha_3^p \in \gamma_2(G)$. Let $M := \gen{\alpha_1, \alpha_2}$. Then the nilpotency class of $M$ and $G$ are equal,  $\gamma_2(G) = \gamma_2(M)$,  $|M| = p^6$, $\Z(M) \le \gamma_2(M)$ and $|\Z(G)/\Z(M)| = p$. Notice that $M$ and $G$ are isoclinic. Again invoking  \cite[Theorem A]{MR}, $\K(M) \ne \gamma_2(M)$ if and only if the nilpotency class of $M$ is $4$ and  $\Z(M)$ is of order $p^2$. One can now easily conclude that $\K(G) \neq \gamma_2(G)$ if and only if the nilpotency class of $G$ is $4$ and  $|\Z(G)| = p^3$. That the commutator length of $G$ is at most $2$  follows from \cite[Theorem A]{MR}, which completes the proof. \hfill $\Box$
 
 \end{proof}

We are now ready to prove Theorem A.

\noindent {\it Proof of Theorem A.}   If $d(\gamma_2(G)) \leq 3$,  then $\K(G)=\gamma_2(G)$ by \cite[Theorem A]{iH20}.  Also when the nilpotency class of $G$ is $6$, then $\K(G) = \gamma_2(G)$ by Theorem \ref{max}. This proves assertion (1). So we only need to consider  $d(\gamma_2(G)) \ge 4$ and the nilpotency class of $G$  at most $5$ .  If $|\gamma_2(G)| =p^4$,   then   assertion (2)  follows from Lemma \ref{lastlem}.  The only case which remains is  $|\gamma_2(G)| = p^5$.  Notice that the nilpotency class of $G$ is at least $4$. We now go according to the nilpotency class of $G$. If the nilpotency class of $G$ is $4$, then  $\gamma_2(G)$ is abelian by Lemma \ref{nab1}, and therefore it is elementary abelian by Lemma \ref{clsatlst5}. Assertion (4) now follows from Theorem \ref{eabcls4}. 

Next assume that the nilpotency class of $G$ is $5$. If  $\gamma_2(G)$ is elementary ablelian, then assertion (5) follows from Proposition \ref{eabcls5}. If it is non-abelian of exponent $p$, then assertion (6) follows from   Proposition \ref{nab}. Finally if the exponent of $\gamma_2(G)$ is $p^2$, then assertion (3) follows from Lemma \ref{cls5}. The proof is now complete.  \hfill $\Box$

\vspace{.2in}

We conclude by exhibiting various examples of groups of order $p^7$ to show that no class of groups considered in Theorem A is void. These examples are constructed from the structural information of the groups obtained in  the preceding two sections, and have been verified for $p= 5, 7$ using GAP \cite{GAP}. For notational convenience, we use long generating set instead of minimal one.
The following two groups satisfy the hypotheses of Lemma \ref{cls5}:
  \begin{eqnarray*}
 G &=& \Big\langle  \alpha_1,  \alpha_2, \alpha_3, \alpha_4, \alpha_5, \alpha_6, \alpha_7  \mid\ [\alpha_1,  \alpha_2]=\alpha_3,  [\alpha_3, \alpha_1]=\alpha_4,  [\alpha_3,  \alpha_2]= \alpha_5, [\alpha_4,  \alpha_1]=\alpha_6,   \\
 & & \;\ [\alpha_6, \alpha_1]=[\alpha_5, \alpha_2] = \alpha_7 = \alpha_3^p,  \alpha_1^p=\alpha_5, \alpha_2^p=\alpha_6,  \alpha{_i}^p=1~ (4 \le i \le 7) \Big\rangle.
\end{eqnarray*}
For this group $G$,  $\gamma_2(G)$  is abelian of exponent $p^2$,  $|\Z(G)|=p$ and $\K(G) \neq \gamma_2(G)$.

  \begin{eqnarray*}
 G &=& \Big\langle  \alpha_1,  \alpha_2, \alpha_3, \alpha_4, \alpha_5, \alpha_6, \alpha_7  \mid\ [\alpha_1,  \alpha_2]=\alpha_3,  [\alpha_3, \alpha_1]=\alpha_4, [\alpha_3,  \alpha_2]= \alpha_5, [\alpha_4,  \alpha_1]=\alpha_6,  \\
 & & \;\ [\alpha_6, \alpha_2]=[\alpha_4, \alpha_3] = [\alpha_1, \alpha_5]= \alpha_7= \alpha_3^p,  \alpha_1^p=\alpha_6, \alpha_2^p=\alpha_5,  \alpha{_i}^p=1~ (4 \le i \le 7) \Big\rangle.
\end{eqnarray*}
This is an example when  $\gamma_2(G)$ is non-abelian of exponent $p^2$,  $|\Z(G)|=p$ and $\K(G) \neq \gamma_2(G)$.

\vspace{.2in}

The next example satisfies the hypotheses of  Proposition \ref{eabcls5}.  
   \begin{eqnarray*}
 G &=& \Big\langle  \alpha_1,  \alpha_2, \alpha_3, \alpha_4, \alpha_5, \alpha_6, \alpha_7  \mid\ [\alpha_1,  \alpha_2]=\alpha_3,  [\alpha_3, \alpha_1]=\alpha_4,  [\alpha_3,  \alpha_2]= \alpha_5, \\
 & & \;\    [\alpha_4,  \alpha_1]=\alpha_6, [\alpha_6,  \alpha_1]= \alpha_7,    \alpha{_i}^p=1~ (1 \le i \le 7) \Big\rangle.
\end{eqnarray*}
For this group $G$, $\gamma_2(G)$  is elementary abelian,  $|\Z(G)|=p^2$ and $\K(G) \neq \gamma_2(G)$

\vspace{.2in}

The following two groups satisfy the hypotheses of Proposition \ref{nab}:
  \begin{eqnarray*}
 G &=& \Big\langle  \alpha_1,  \alpha_2, \alpha_3, \alpha_4, \alpha_5, \alpha_6, \alpha_7  \mid\  [\alpha_1,  \alpha_2]=\alpha_3,  [\alpha_3, \alpha_1]=\alpha_4,  [\alpha_3,  \alpha_2]= \alpha_5,  \\
 & & \;\ [\alpha_4,  \alpha_1]=\alpha_6,   [\alpha_6, \alpha_2]=[\alpha_4, \alpha_3]= [\alpha_4,  \alpha_2]= \alpha_7,    \alpha{_i}^p=1~ (1 \le i \le 7) \Big\rangle.
\end{eqnarray*}
Notice that $\gamma_2(G)$  is non-abelian of exponent $p$,  $|\Z(G)|=p^2$ and $\K(G) \neq \gamma_2(G)$ for this group.

    \begin{eqnarray*}
 G &=& \Big\langle  \alpha_1,  \alpha_2, \alpha_3, \alpha_4, \alpha_5, \alpha_6, \alpha_7  \mid\  [\alpha_1,  \alpha_2]=\alpha_3,  [\alpha_3, \alpha_1]=\alpha_4,  [\alpha_5,  \alpha_1]=\alpha_6 \alpha_7, \\
 & & \;\ [\alpha_3,  \alpha_2]= \alpha_5,  [\alpha_4,  \alpha_2]= \alpha_6,      [\alpha_6, \alpha_1]=[\alpha_3, \alpha_4]= \alpha_7,     \alpha{_i}^p=1~ (1 \le i \le 7) \Big\rangle.
\end{eqnarray*}
This is a group with  $\gamma_2(G)$  non-abelian of exponent $p$,  $|\Z(G)|=p$ and $\K(G) = \gamma_2(G)$.

\vspace{.2in}

The following two groups are of nilpotency class $4$ and satisfy the hypotheses of Theorem \ref{eabcls4}:
    \begin{eqnarray*}
 G &=& \Big\langle  \alpha_1,  \alpha_2, \alpha_3, \alpha_4, \alpha_5, \alpha_6, \alpha_7  \mid [\alpha_1,  \alpha_2]=\alpha_3,  [\alpha_3, \alpha_1]=\alpha_4,  [\alpha_3,  \alpha_2]=\alpha_5, \\
 & & \;\ [\alpha_4,  \alpha_1]= \alpha_6,  [\alpha_5,  \alpha_2]= \alpha_7,  \alpha{_i}^p=1~ (1 \le i \le 7) \Big\rangle.
\end{eqnarray*}
For this group $G$, we have  $\gamma_2(G)$  is elementary abelian,  $|\Z(G)|=p^2$ and $\K(G) \neq \gamma_2(G)$.

   \begin{eqnarray*}
 G &=& \Big\langle  \alpha_1,  \alpha_2, \alpha_3, \alpha_4, \alpha_5, \alpha_6, \alpha_7  \mid [\alpha_1,  \alpha_2]=\alpha_3,  [\alpha_3, \alpha_1]=\alpha_4,  [\alpha_3,  \alpha_2]= \alpha_5, \\
 & & \;\ [\alpha_5,  \alpha_1]= [\alpha_4,  \alpha_2]=\alpha_7,       [\alpha_4,  \alpha_1]^{\nu}=   [\alpha_5,  \alpha_2]= \alpha_6^{\nu},      \alpha{_i}^p=1~ (1 \le i \le 7) \Big\rangle,
\end{eqnarray*}
where $\nu$ is a non-quadratic residue mod $p$.
Notice that $\gamma_2(G)$   is abelian,  $|\Z(G)|=p^2$ and $\K(G) = \gamma_2(G)$.

\vspace{.2in}

Finally we present two examples which satisfy the hypotheses of Lemma  \ref{lastlem}. These are as follow. 
   \begin{eqnarray*}
 G &=& \Big\langle  \alpha_1,  \alpha_2, \alpha_3, \alpha_4, \alpha_5, \alpha_6, \alpha_7  \mid [\alpha_1,  \alpha_2]=\alpha_4,  [\alpha_4, \alpha_1]=\alpha_5,  [\alpha_4,  \alpha_2]= \alpha_6, \\
 & & \;\ [\alpha_6,  \alpha_1]= [\alpha_5,  \alpha_2]=\alpha_7=\alpha_3^p,       \alpha{_i}^p=1~ (1 \le i \le 7,  i \neq 3) \Big\rangle.
\end{eqnarray*}
Notice that $\gamma_2(G)$   is abelian,  $|\Z(G)|=p^2$ , the nilpotency class of $G$ is $4$ and $\K(G) = \gamma_2(G)$

   \begin{eqnarray*}
 G &=& \Big\langle  \alpha_1,  \alpha_2, \alpha_3, \alpha_4, \alpha_5, \alpha_6, \alpha_7  \mid [\alpha_1,  \alpha_2]=\alpha_4,  [\alpha_4, \alpha_1]=\alpha_5,  [\alpha_4,  \alpha_2]= \alpha_6, \\
 & & \;\  [\alpha_5,  \alpha_1]=\alpha_7=\alpha_3^p,       \alpha{_i}^p=1~ (1 \le i \le 7, i \neq 3) \Big\rangle.
\end{eqnarray*}
For this group $G$,  $\gamma_2(G)$   is abelian,  $|\Z(G)|=p^3$,  the nilpotency class of $G$ is $4$ and $\K(G) \neq \gamma_2(G)$.

Examples of the groups having nilpotency class $3$ and satisfying the hypotheses of Lemma \ref{lastlem} are given in \cite{MR}. One might use GAP for many more examples for adequate primes.

\vspace{.2in}

\noindent{\bf Acknowledgements.} The authors thank Pramath Anamby  for his help in establishing Lemma \ref{quad}.

\end{document}